\documentclass[11pt,a4paper]{article}
\usepackage{amscd}
\usepackage{enumerate}

\usepackage{amsmath, amssymb, latexsym}
\usepackage{amsfonts}
\usepackage{amsthm}
\usepackage{relsize}
\usepackage{setspace}
\usepackage{geometry}
\usepackage{url}
\usepackage{xspace}
\usepackage{tocloft}
\usepackage{graphics}
\usepackage{graphicx}
\usepackage{lscape}
\usepackage{microtype}
\usepackage{ulem}

\usepackage[usenames, dvipsnames]{color}
\usepackage[utf8]{inputenc}
\usepackage{tikz}

\usepackage[pagebackref=true]{hyperref}
\usepackage[alphabetic]{amsrefs}
\usepackage[english]{babel}

\usepackage{authblk}

\newtheorem{theorem}{Theorem}[section]
\newtheorem{proposition}[theorem]{Proposition}

\newtheorem{corollary}[theorem]{Corollary}
\newtheorem{lemma}[theorem]{Lemma}

\theoremstyle{definition}
\newtheorem{definition}[theorem]{Definition}
\newtheorem{example}[theorem]{Example}
\newtheorem{remark}[theorem]{Remark}
\newtheorem{notation}[theorem]{Notation}

\theoremstyle{problem}

\newcommand{\Aut}{\mathrm{Aut}}
\newcommand{\Opp}{\mathrm{Opp}}

\newcommand{\RR}{\mathbf{R}}
\newcommand{\ZZ}{\mathbf{Z}}
\newcommand{\QQ}{\mathbf{Q}}

\newcommand{\Ch}{\mathrm{Ch}}
\newcommand{\St}{\mathrm{St}}
\newcommand{\Stab}{\mathrm{Stab}}
\newcommand{\Fix}{\mathrm{Fix}}
\newcommand{\bd}{\partial}
\newcommand{\diag}{\operatorname{diag}}

\newcommand{\id}{\operatorname{id}}

\newcommand{\SO}{\operatorname{SO}}

\newcommand{\SL}{\operatorname{SL}}

\newcommand{\dist}{\operatorname{dist}}

\def\og{\leavevmode\raise.3ex\hbox{$\scriptscriptstyle\langle\!\langle$~}}
\def\fg{\leavevmode\raise.3ex\hbox{~$\!\scriptscriptstyle\,\rangle\!\rangle$}}

\title{A unified proof of the Howe--Moore property}

\author[1]{Corina Ciobotaru\thanks{Supported  by the FRIA; corina.ciobotaru@uclouvain.be}}

\date{First draft: March 2, 2014; Accepted: May 20, 2014}

\begin{document}

\maketitle

\begin{abstract}
We provide a unified proof of all known examples of locally compact 
groups that enjoy the Howe--Moore property, namely, 
the vanishing at infinity of all matrix coefficients 
of the group unitary representations that are without 
non-zero invariant vectors. These examples are: 
connected, non-compact, simple real Lie groups with finite center,  
isotropic simple algebraic groups over non Archimedean 
local fields and closed, topologically simple subgroups 
of $\Aut(T)$ that act $2$--transitively on the boundary 
$\partial T$, where $T$ is a bi-regular tree with valence 
$\geq 3$ at every vertex.  
\end{abstract}

\renewcommand{\thefootnote}{}
\newcounter{qcounter}

\section{Introduction}

For a locally compact group $G$, the Howe--Moore property asserts that all matrix coefficients of its unitary representations, that are without non-zero $G$--invariant vectors, vanish at infinity. This property was first established by Howe and Moore~\cite{HM79} and Zimmer~\cite{Zim84}, around 1977, for connected, non-compact, simple real Lie groups that are with finite center. It plays an important role in the Mostow rigidity theorem, as well as in various other rigidity results. The original proof of the Howe--Moore property uses the geometry of semi-simple real Lie groups. Nowadays the proof is more algebraic (see for example Bekka and Mayer~\cite{BM} or Morris~\cite[Theorem~10.14]{WM} for the case of $\SL(2,\mathbb{R})$). 

In the same article, Howe and Moore also treat the case of isotropic simple algebraic groups over non Archimedean local fields (see Howe and Moore~\cite[Theorem~5.1]{HM79} and Definition~\ref{def::alg_groups}). By Bruhat and Tits~\cite{BT72}, to such a group one associates a locally finite thick Euclidean building $\Delta$ where the group acts continuously and properly, by type-preserving automorphisms and strongly transitively (for strong transitivity see Definition~\ref{def::strong_trans_action}). For other totally disconnected analogs of semi-simple real Lie groups, namely, closed, strongly transitive subgroups of $\Aut(\Delta)$, where $\Delta$ is a locally finite thick Euclidean building, the study of the Howe--Moore property was initiated by Lubotzky and Mozes~\cite{LM92} in the particular case of $\Delta$ being a $d$--regular tree $T_d$, with $d \geq 3$. Using the geometry of horocycles and horoballs, they prove that the index-two subgroup of $\Aut(T_{d})$, which preserves the $2$--coloring of the tree $T_{d}$, enjoys this property. A more general result, in the context of $d$--regular trees, was obtained by Burger and Mozes~\cite{BM00a}. There they show that every closed, topologically simple subgroup of $\Aut(T_{d})$ which is $2$--transitive on the boundary has the Howe--Moore property. We mention that the latter class of subgroups of $\Aut(T_d)$ contains also non linear examples. For completeness, we add the known fact that in the case of a thick tree $T$, the strongly transitive action of a subgroup of $\Aut(T)$ on $T$ is equivalent to the $2$--transitive action of that group on the boundary of $T$ (see for example Caprace and Ciobotaru~\cite[Theorem~1.1 and Corollary~3.6]{CaCi}); in particular, the latter mentioned equivalence implies that $T$ is a bi-regular tree.

\medskip
Regarding all examples presented above, this article proposes to give a unified proof of the Howe--Moore property:

\begin{theorem}(See Theorem~\ref{thm::H-M-buildings} and Section~\ref{subsec::real_Lie_groups}) 
\label{thm::H-M-buildings_2}
Let $G$ be a connected, non-compact, simple real Lie group, with finite center, or an isotropic simple algebraic group over a non Archimedean local field, or a closed, topologically simple subgroup of $\Aut(T)$ that acts $2$--transitively on the boundary $\partial T$, where $T$ is a bi-regular tree with valence $ \geq 3$ at every vertex. Then $G$ admits the Howe--Moore property.
\end{theorem}

The main ingredients of the unified proof, are: 
\begin{list}{\roman{qcounter})~}{\usecounter{qcounter}}
\item
the \textbf{$K_1A^{+}K_2$ decomposition} of these groups, where $K_1,K_2$ are compact subsets and $A^{+}$ is the `maximal abelian' sub semi-group part (see Section \ref{sec::unif_examples});
\item 
the \textbf{`contraction'} subgroups $U^{\pm}_{\alpha}$ (see Definition~\ref{def::contr_groups}) with respect to hyperbolic elements, combined with basic facts regarding normal operators (see Sections \ref{subsec::Mautner}, \ref{subset::normal_op_weak_lim}). 
\end{list} More precisely, the above mentioned ingredients are assembled to obtain the following criterion, used to verify the Howe--Moore property:

\begin{theorem}(See Theorem~\ref{Crit_H_M})
\label{thm::Crit_H_M}
Let $G$ be a second countable, locally compact group having the following properties:
\begin{list}{\roman{qcounter})~}{\usecounter{qcounter}}
\item
$G$ admits a decomposition $K_1A^{+}K_2$, where $K_1,K_2 \subset G$ are compact subsets and $A^{+}$ is an abelian sub semi-group of $G$;
\item 
every sequence $\alpha=\{g_{n}\}_{n>0} \subset A^{+}$, with $g_{n} \rightarrow \infty$,  admits a subsequence $\beta=\{g_{n_k}\}_{n_k>0}$ such that  $G= \overline{ \left\langle U^{+}_{\beta},U^{-}_{\beta}\right\rangle }$.
\end{list}

Then $G$ has  the Howe--Moore property. 
\end{theorem}

The idea behind this criterion is inspired by a proof given for $\SL (2, \mathbb{R})$ by Dave Witte Morris~\cite[Theorem~10.14]{WM}.

To the best of our knowledge, there are no other known examples of locally compact or even discrete groups enjoying the Howe--Moore property, beside those covered by the unified proof given in this article. Some necessary algebraic conditions for a locally compact group admitting the Howe--Moore property are described in Cluckers--de Cornulier--Louvet--Tessera--Valette~\cite{CCL+}. Those conditions are quite elementary, but probably far from being sufficient. Even for the two classes of totally disconnected locally compact groups of Theorem~\ref{thm::H-M-buildings_2}, we still do not know to what extent the Howe--Moore property relates with the $KA^{+}K$ decomposition and with the strong transitivity on the corresponding Euclidean building; in this direction, only partial results have been obtained so far.

\section{Preliminaries}

\subsection{Basic definitions}

In what follows all Hilbert spaces are assumed to be complex, not necessarily separable, and their inner product is denoted by $\left\langle \cdot ,  \cdot \right\rangle $. Let $B(\mathcal{H})$ be the space of all bounded linear operators $ F: \mathcal{H} \rightarrow \mathcal{H}$. 

Recall that the closed unit ball of a Hilbert space  $\mathcal{H}$ is compact with respect to the weak topology and moreover, by the Eberlein--\v Smulian Theorem, compactness is equivalent to sequential compactness. Hence, any sequence of vectors $\{v_n\}_{n\geq 0} \subset \mathcal{H}$, with their norms being bounded above, contains a \textbf{weakly convergent} subsequence $ \{v_{n_{k}}\}_{k>0}$ in $\mathcal{H}$. Namely, there exists a vector $v \in \mathcal{H}$ such that for any $w \in \mathcal{H}$ one has that $\lim\limits_{n_k \rightarrow \infty} \left\langle v_{n_{k}} , w \right\rangle =\left\langle v , w \right\rangle$.

In the same spirit, if $\mathcal{H}$ is a separable Hilbert space, it is well-known that the closed unit ball in $B(\mathcal{H})$ is compact and metrizable with respect to the weak operator topology. Thus, any sequence of operators $\{F_n\}_{n\geq 0} \subset B(\mathcal{H})$, with their operator norms being bounded above, contains a \textbf{weakly convergent} subsequence $ \{F_{n_{k}}\}_{k>0}$ in $B(\mathcal{H})$. Namely, there exists $F$ in $B(\mathcal{H})$ such that for all $v,w \in \mathcal{H}$ one has that $\lim\limits_{n_{k} \rightarrow \infty} \left\langle F_{n_{k}}v , w \right\rangle =\left\langle Fv , w \right\rangle$

\begin{definition}
Let $\mathcal{H}$ be a Hilbert space. We say that $E \in B(\mathcal{H})$ is \textbf{normal} if $E^{*}E=EE^{*}$. It is well-known that this is equivalent to $\left\langle EE^{*}v, v \right\rangle =\left\langle E^{*}Ev, v \right\rangle $, for every $v \in \mathcal{H}$. Moreover, $U \in B(\mathcal{H})$ is called \textbf{unitary} if $UU^{*}=U^{*}U=I$ or, equivalently, if $\left\langle U \xi,U \eta \right\rangle =\left\langle \xi, \eta \right\rangle $ for all $\xi, \eta \in \mathcal{H}$ and $U$ is onto. We denote by $\mathcal{U}(\mathcal{H})$ the group of all unitary operators of $\mathcal{H}$.
\end{definition}


\begin{definition}
A \textbf{unitary representation} of a topological group $G$ into a Hilbert space $\mathcal{H}$ is a group homomorphism $\pi : G \rightarrow \mathcal{U}(\mathcal{H})$, which is moreover strongly continuous: the map $g \in G \mapsto \pi(g) v \in \mathcal{H}$ is continuous for every $v \in \mathcal{H}$. Instead of $\pi : G \rightarrow \mathcal{U}(\mathcal{H})$ we often write $(\pi, \mathcal{H})$ for a unitary representation of $G$.

Furthermore, any two vectors $v, w \in \mathcal{H}$ define a continuous function $c_{v, w} : G \rightarrow \mathbb{C}$, given by $c_{v, w}(g): = \left\langle \pi(g)v , w \right\rangle$ and we call  it the associated \textbf{$( v, w)$--matrix coefficient}. Moreover, we say that $c_{v, w}$ \textbf{is} $\mathbf{\mathrm{C}_0}$ if for every $\epsilon >0$, the subset $\{ g \in G \; : \; |c_{v, w}(g)| \geq \epsilon\}$ is compact in $G$.
\end{definition}

\begin{remark}(See Bekka, de la Harpe and Valette~\cite[Proposition~C.4.9]{BHV})
\label{rem::cyclic_rep}
We mention that for a unitary representation $(\pi, \mathcal{H})$ of a topological group $G$, the Hilbert space $\mathcal{H}$ can be decomposed as a direct sum $\mathcal{H}=\bigoplus_{i} \mathcal{H}_{i}$ of mutually orthogonal, closed and $G$--invariant subspaces $\mathcal{H}_{i}$, such that the restriction of $\pi$ to $\mathcal{H}_{i}$ is cyclic for every $i$ (i.e., for every $i$ there exists a non-zero vector $v \in \mathcal{H}_{i}$ such that $\pi(G)v$ is dense in $\mathcal{H}_{i}$).
\end{remark}

Moreover, the following well-known lemma asserts that for `nice' topological groups, only separable Hilbert spaces can be considered when working with unitary representations. For the convenience of the reader, we include its proof.
\begin{lemma}
\label{lem::sep_loc_compact}
Let $G$ be a separable locally compact group. Then for every cyclic unitary representation $(\pi, \mathcal{H})$ of $G$, the corresponding Hilbert space $ \mathcal{H}$ is separable. In particular, all irreducible unitary representations of $G$ are over separable Hilbert spaces.
\end{lemma}

\begin{proof}
Let $(\pi, \mathcal{H})$ be a cyclic unitary representation of $G$ and denote by $v$ a cyclic vector of it. Then, we have that $f: G \rightarrow \mathcal{H}$, defined by $f(g):= \pi(g)v$, is a continuous function. Moreover, by the definition of a cyclic vector, we know that the linear span of $\pi(G)v$ is dense in $\mathcal{H}$.

As $G$ is separable, let $Q$ be a countable dense subset of it. Thus $f(Q)$ is dense in $\pi(G)v$ and therefore the linear span of $\pi(Q)v$ is  also dense in $\mathcal{H}$. The conclusion thus follows.
\end{proof}

\begin{remark}
\label{rem::exam_sep_loc_com}
The class of separable locally compact groups contains non-compact, simple real Lie groups with finite center and also closed, non-compact subgroups of $\Aut(\Delta)$, where $\Delta$ is a locally finite thick Euclidean building. In particular, all non-compact, simple $p$--adic Lie groups are separable.
\end{remark}

\subsection{The Howe--Moore property}

\begin{definition}
Let $G$ be a locally compact group and let $G \bigcup \{\infty\}$ be the one point compactification of $G$. We say that $G$ has the \textbf{Howe--Moore property} if the set of all unitary representations of $G$ is the union of the ones having non-zero $G$--invariant vectors and the ones for which all matrix coefficients are $\mathrm{C}_0$. When a matrix coefficient is $\mathrm{C}_0$, we say also that it \textbf{vanishes at $\infty$}. 
\end{definition}

Throughout this article we assume that all locally compact groups are second countable; this is to simplify our notation, by using sequences instead of nets. The second reason is given by the next remark.

\begin{remark}
\label{rem::irr_H-M}
By~\cite[Prop.~2.3]{CCL+}, to prove the Howe--Moore property for a second countable, locally compact group $G$ it is enough to verify the $C_0$ condition only for all irreducible, non-trivial, unitary representations of $G$, which, by Lemma~\ref{lem::sep_loc_compact}, are over separable Hilbert spaces.

\end{remark}

Moreover, the following well-known lemma shows that matrix coefficients not vanishing at $\infty$ give rise to particular matrix coefficients with the same property. 
\begin{lemma}
\label{v=w}
Let $G$ be a second countable, locally compact group and $(\pi, \mathcal{H})$ be a unitary representation of $G$. Suppose there exist two non-zero vectors $v, w \in \mathcal{H}$ such that the $( v, w )$--matrix coefficient does not vanish at $\infty$. Then the $( v, v)$--matrix coefficient does not vanish at $\infty$.
\end{lemma}

\begin{proof}
By hypothesis, there exists a sequence $\{g_n\}_{n\geq 0} \subset G$ and $\epsilon >0$ such that $g_n \to \infty$ and $\vert \left\langle \pi(g_{n})v, w \right\rangle \vert \geq \epsilon$. Thus $w$ is not orthogonal to the Hilbert subspace generated by $\pi(G)v$. Denote this Hilbert subspace by $\overline{ \langle \pi(G)v \rangle }$ and remark this is $G$--invariant.

Moreover, $\{\pi(g_n) v\}_{n\geq 0}$ being bounded in the norm of $\mathcal{H}$, there exist $v_0 \in  \mathcal{H}$ and a subsequence $\{n_k\}_{k\geq 0}$ such that $\{\pi(g_{n_k}) v\}_{k\geq 0}$ weakly converges to $v_0$. Thus $\lim\limits_{n_k \rightarrow \infty} \left\langle \pi(g_{n_k}) v , w'  \right\rangle =\left\langle v_0 , w' \right\rangle$, for every $w' \in  \mathcal{H}$. From here we conclude that $v_0$ is a non-zero vector. 

Furthermore, we claim that $v_0$ is a vector in $\overline{ \langle \pi(G)v \rangle }$. Indeed, suppose the converse that $v_0=w_1+w_2$, with $w_1 \in \overline{ \langle \pi(G)v \rangle }$ and $w_2$ a non-zero vector orthogonal to $\overline{ \langle \pi(G)v \rangle }$. Taking $w'=w_2$, we obtain that $\lim\limits_{n_k \rightarrow \infty} \left\langle \pi(g_{n_k}) v , w_2  \right\rangle =0=\left\langle v_0 , w_2 \right\rangle$, which is impossible. The claim follows and we conclude that there exists $g \in G$ such that $\left\langle v_0 , \pi(g)v \right\rangle \neq 0$. For this $g$ we obtain that $\lim\limits_{n_k \rightarrow \infty} \left\langle \pi(g_{n_k}) v , \pi(g)v \right\rangle =\left\langle v_0 , \pi(g)v \right\rangle$. The conclusion follows by  only remarking that $\vert \left\langle \pi(g^{-1}g_{n_k})v,  v \right\rangle \vert \nrightarrow 0$.
\end{proof}

In addition, the following easy lemma asserts that compact subsets do not count for the Howe--Moore property, given a `polar decomposition' of the locally compact group.

\begin{lemma}
\label{sequn_A}
Let $G$ be a second countable, locally compact group admitting a decomposition $G=K_{1}AK_{2}$, where $K_{1},K_{2}$ are compact subset and $A$ is any subset of $G$. Let $(\pi, \mathcal{H})$ be a unitary representation of $G$. 
If for every sequence $\{a_{n}\}_{n>0} \subset A$, with $\{a_{n}\}_{n>0} \rightarrow \infty $, the corresponding matrix coefficients are $C_0$, then all matrix coefficients of $G$, with respect to $(\pi, \mathcal{H})$, vanish at $\infty$.
\end{lemma}

\begin{proof}[Proof]
Suppose there exist a sequence $\{g_{n}\}_{n>0} \subset G $, with $g_{n} \rightarrow \infty $, and two non-zero vectors $v,w \in \mathcal{H}$ such that the corresponding matrix coefficient $c_{v,w}(g_{n})$ does not vanish at $\infty$. As $G=K_{1}AK_{2}$ one has $g_{n}=k_{n}a_{n}h_{n}$, where $k_{n} \in K_1, h_{n} \in K_2$ and $a_{n} \in A$, for every $n>0$. We have thus $a_{n} \rightarrow \infty $. As $K_{1},K_{2}$ are compact sets and by passing to a subsequence, we can assume that $\{ \pi(k_{n})^{-1}w \}_{n>0}$ and $\{ \pi(h_{n})v \}_{n>0}$ converge, in the norm of $\mathcal{H}$, to the vectors $w_{0}$, respectively $v_{0}$. One has:
\begin{equation*}
\begin{split}
\vert \left\langle \pi (k_{n}a_{n}h_{n})v,w \right\rangle \vert &= \vert \left\langle \pi (a_{n}h_{n})v,\pi(k_{n})^{-1}w \right\rangle -\left\langle \pi (a_{n}h_{n})v,w_{0} \right\rangle \\
&+ \left\langle \pi (h_{n})v,\pi(a_{n})^{-1} w_{0} \right\rangle- \left\langle v_{0},\pi(a_{n})^{-1} w_{0} \right\rangle + \left\langle \pi(a_{n}) v_{0}, w_{0} \right\rangle \vert \\
& \leq \Vert \pi(k_{n})^{-1}w-w_0 \Vert \cdot \Vert v \Vert + \Vert (h_{n})v -v_0\Vert \cdot \Vert  w_0\Vert + \vert \left\langle \pi(a_{n}) v_{0}, w_{0} \right\rangle \vert. \\
\end{split}
\end{equation*} As the first two terms tend to zero as $n \rightarrow \infty$, we have that the matrix coefficient $c_{v_{0},w_{0}}(a_{n})$ does not vanish at $\infty$. We obtain a contradiction and the conclusion follows.
\end{proof}

Let us add here some necessary algebraic conditions for a non-compact locally compact group to admit the Howe--Moore property; however, those conditions are far from being sufficient.  

\begin{proposition}(See Prop.~3.2 in~\cite{CCL+})
\label{prop::open_compact}
Let $G$ be a non-compact locally compact group with the Howe--Moore property. Then any proper open subgroup of $G$ is compact and thus of infinite index in $G$.
\end{proposition}

\subsection{Mautner's  phenomenon}

\label{subsec::Mautner}
One of the main ideas to verify the Howe--Moore property is to search for `big' subgroups $H \leq G$ admitting at least one non-zero $H$--invariant vector in $\mathcal{H}$. Mautner's phenomenon gives us a method to construct such subgroups. First, let us introduce the following: 

\begin{definition}
\label{def::contr_groups}
Let $G$ be a locally compact group. Let $\alpha=\{g_{n}\}_{n>0}$ be a sequence in $G$. This sequence $\alpha$ defines the set:
\[ 
U^{+}_{\alpha}:= \{ g \in G \; | \; \lim\limits_{n \to \infty} g_{n}^{-1}gg_{n}=e \}. 
\]

Notice that $U^{+}_{\alpha}$ is a subgroup of $G$. It is called \textbf{the contraction group} corresponding to $\alpha$. Because it doesn't need to be closed in general, we denote by $N^{+}_{\alpha}$ the closed subgroup $\overline{U^{+}_{\alpha}}$. In the same way, but using $g_{n}gg_{n}^{-1}$ we define $U^{-}_{\alpha}$ and $N^{-}_{\alpha}$. When $\alpha=\{a^{n}\}_{n>0}$, for some $a \in G$, we simply use the notation $U^{\pm}_{a}:=U^{\pm}_{\alpha}$ and $N^{\pm}_{a}:=N^{\pm}_{\alpha}$.
\end{definition}

 We record the following lemma, known as Mautner's phenomenon (see Bekka--Mayer~\cite[Chp.~III, Thm.~1.4]{BM}): 

\begin{lemma}
\label{Mautner}
Let $G$ be a second countable, locally compact group and $(\pi, \mathcal{H})$ be a unitary representation. Let $\{g_{n}\}_{n>0}$ be a sequence in $G$ and take $v \in \mathcal{H}$ a non-zero vector. Suppose moreover that $\{\pi(g_{n})v\}_{n>0}$ weakly converges to $v_{0}$. 

Then $v_{0}$ is $N^{+}_{\alpha}$--invariant.
\end{lemma}

\begin{proof}
Let $w \in \mathcal{H}$ be an arbitrary vector and take $g \in U^{+}_{\alpha}$. We have:
\begin{equation*}
\begin{split}
|\left\langle \pi(g)v_{0}-v_{0}, w \right\rangle |&=\lim_{n \rightarrow \infty}|\left\langle \pi(g g_{n})v-\pi(g_{n})v, w \right\rangle |\\
&=\lim_{n \rightarrow \infty}|\left\langle \pi(g_{n}^{-1}g g_{n})v-v, \pi(g_{n}^{-1})w \right\rangle |\\
&\leq \lim_{n \rightarrow \infty}\Vert (\pi(g_{n}^{-1}g g_{n})-Id_{\mathcal{H}})v \Vert \cdot \Vert w\Vert=0.\\
\end{split}
\end{equation*}

Since $U^{+}_{\alpha}$ is dense in $N^{+}_{\alpha}$ and $\pi$ is continuous, it follows that $N^{+}_{\alpha}$ fixes $v_{0}$.
\end{proof}

\subsection{Normal operators as weak limits}
\label{subset::normal_op_weak_lim}

This elementary lemma is a key step used in Section \ref{sec::alg_crit}. 

\begin{lemma}
\label{vN_normal}
Let $\mathcal{H}$ be a separable Hilbert space and consider a sequence $\{ U_{n} \}_{n>0} \subset \mathcal{U}(\mathcal{H})$ of pairwise commuting unitary operators.

Then there exists a subsequence $\{ U_{n_{k}} \}_{n_{k}}$  which weakly converges to an operator $E \in B(\mathcal{H})$. Moreover, any such weak limit is normal.
\end{lemma}

\begin{proof}[First proof]
As $\{ U_{n} \}_{n>0}$ are unitary operators, they have norm one in $B(\mathcal{H})$. Hence, there is a subsequence $\{ U_{n_{k}} \}_{n_{k}}$  which  weakly converges to an operator $E \in B(\mathcal{H})$. It remains to prove that $E$ is normal. In what follows we replace the subsequence $\{ U_{n_{k}} \}_{n_{k}}$ by the sequence $\{ U_{n} \}_{n>0}$. Notice that, as the sequence $\{ U_{n} \}_{n>0}$ converges weakly to $E$, the sequence  $\{ U_{n}^{*} \}_{n>0}$ converges weakly to $E^{*}$.


By hypothesis, the unitary operators $\{ U_{n} \}_{n>0}$ commute pairwise, thus $U_{n}$ and $U_{k}^{*}$ commute for all $n,k>0$. Therefore for all $v,w \in \mathcal{H}$ one has: 
\begin{equation*}
\begin{split}
\left\langle EE^{*}v, w \right\rangle &= \lim_{n \rightarrow \infty}\left\langle U_{n}(E^{*} v), w \right\rangle = \lim_{n \rightarrow \infty} \left\langle E^{*} v, U_{n}^{*} w \right\rangle \\
&= \lim_{n \rightarrow \infty}( \lim_{k \rightarrow \infty} \left\langle U_{k}^{*} v , U_{n}^{*} w \right\rangle)= \lim_{n \rightarrow \infty}( \lim_{k \rightarrow \infty} \left\langle U_{n} U_{k}^{*} v,  w \right\rangle)\\
&= \lim_{n \rightarrow \infty}( \lim_{k \rightarrow \infty} \left\langle U_{k}^{*} U_{n} v, w \right\rangle)= \lim_{n \rightarrow \infty} \left\langle U_{n} v , E w \right\rangle =\left\langle  E v, E w \right\rangle\\
&= \left\langle E^{*}  E v, w \right\rangle.\\  
\end{split} 
\end{equation*}
\end{proof}

\begin{proof}[Second proof]
Another way to see this is using von Neumann algebras. More precisely, as the unitary operators $\{ U_{n}\}_{n>0}$ commute pairwise, they generate a commutative von Neumann algebra which is closed, by definition, in the weak topology. Thus the weak limit $E$ is also contained in this von Neumann algebra. We conclude that $E$ is normal as operator.
\end{proof}

The pairwise commutativity condition in Lemma~\ref{vN_normal} is essential to obtain a normal weak limit. The following is an example in that sense.
\begin{example}
\label{exa::not_normal}
Consider the Hilbert space $ \mathcal{H}:=\ell^{2}(\mathbb{Z})\bigoplus \ell^{2}(\mathbb{Z})$ and for every $n >0$, let $U_{n}:=\left( \begin{array}{cc} 0 & \lambda_{n} \\ 1 & 0 \end{array} \right) $, where $\lambda_n$ is the unitary operator of translation by $n$ of the Hilbert space $\ell^{2}(\mathbb{Z})$. Notice that the sequence of unitary operators $\{\lambda_n\}_{n>0}$ weakly converges to $0$. Therefore, $\{ U_{n} \}_{n>0}$ is a sequence of unitary operators of $\mathcal{H}$ which do not commute pairwise and whose weak limit is the operator $E=\left( \begin{array}{cc} 0 & 0 \\ 1 & 0 \end{array}  \right)$. Equally, the weak limit of $\{ U_{n}^{*} \}_{n>0}$ is the operator $E^{*}=\left( \begin{array}{cc} 0 & 1 \\ 0& 0 \end{array} \right)$ and one can notice that $EE^{*}=\left( \begin{array}{cc} 0 & 0 \\ 0 & 1 \end{array} \right)$  and $E^{*}E=\left( \begin{array}{cc} 1 & 0 \\ 0 & 0 \end{array} \right)$. 
Therefore, $E$ is not normal.
\end{example}

\section{An algebraic criterion}
\label{sec::alg_crit}

This section proposes a criterion, given by Theorem~\ref{Crit_H_M}, to verify the Howe--Moore property, over separable Hilbert spaces, for locally compact groups admitting a special decomposition. Its proof uses the idea of Dave Witte Morris~\cite[Theorem~10.14]{WM}, which relies on a more elaborate version of Mautner's phenomenon; for further references this idea is recorded in a separate lemma.
\begin{lemma}
\label{lem::gen_G_invariant_vect}
Let $G$ be a second countable, locally compact group and $(\pi, \mathcal{H})$ be a unitary representation of $G$, without $G$--invariant vectors. Assume there exists a sequence $\alpha=\{g_{n}\}_{n>0}$ of pairwise commuting elements of $G$ such that $g_n \to \infty $ and $G= \overline{ \left\langle U^{+}_{\alpha},U^{-}_{\alpha}\right\rangle }$.

Then all matrix coefficients of $(\pi, \mathcal{H})$ corresponding to $\{g_{n}\}_{n>0}$ are $C_0$.
\end{lemma}

\begin{proof}
As $G$ is second countable and using Lemma~\ref{lem::sep_loc_compact}, we can assume, without loss of generality, that the Hilbert space $\mathcal{H}$ is separable. In particular, by restricting to a subsequence, we can consider  that  the sequence $\{\pi(g_n)\}_{n>0}$ weakly converges to an operator $E$.

By contraposition, we assume there exist two non-zero vectors $v,w \in \mathcal{H}$ (considered fixed for what follows) such that $ \vert  \left\langle \pi(g_{n})v, w \right\rangle \vert  \nrightarrow 0$. We want to prove that $(\pi, \mathcal{H})$ has a non-zero $G$--invariant vector. Because $\{\pi(g_n)\}_{n>0}$ weakly converge to an operator $E$, by  Lemma~\ref{vN_normal} we have moreover that $E$ is normal and $E \neq 0$.

\medskip
We now claim that the following important property of $E$ holds: $\pi(u_{+})E\pi(u_{-})=E$, for every $u_{+} \in U^{+}_{\alpha}$ and every $u_{-} \in U^{-}_{\alpha}$. 

\textit{Proof of the claim.}

Let $u_{-} \in U^{-}_{\alpha}$.  For every $\phi, \psi \in \mathcal{H}$ one has:
\begin{equation*}
\begin{split}
|\left\langle E\pi(u_{-})\phi, \psi \right\rangle- \left\langle E\phi, \psi \right\rangle |&= |\lim_{n \rightarrow \infty} \left\langle \pi(g_{n} u_{-})\phi-\pi(g_{n})\phi, \psi \right\rangle|\\
&= |\lim_{n \rightarrow \infty} \left\langle \pi(g_{n} u_{-} g_{n}^{-1}) \pi(g_{n})\phi- \pi(g_{n})\phi, \psi \right\rangle |\\
&=|\lim_{n \rightarrow \infty} \left\langle \pi(g_{n})\phi, \pi(g_{n} u_{-}^{-1} g_{n}^{-1}) \psi-\psi \right\rangle | \\
&\leq \lim_{n \rightarrow \infty} \Vert \pi(g_{n})\phi\Vert \cdot \Vert (\pi(g_{n} u_{-}^{-1} g_{n}^{-1}-e)\psi\Vert =0, \\
\end{split}
\end{equation*} as $\lim\limits_{n \rightarrow \infty} g_{n}u_{-}g_{n}^{-1}=e$. We conclude that $E\pi(u_{-})=E$, for every $u_{-} \in U^{-}_{\alpha}$.
In the same way one proves that $\pi(u_{+})E=E$, for every $u_{+} \in U^{+}_{\alpha}$. This proves the claim.

\medskip

By the claim we obtain that,  for every $\phi \in \mathcal{H}$, the vector $E(\phi)$ (respectively, $E^{*}(\phi)$) is $U^{+}_{\alpha}$ (respectively, $U^{-}_{\alpha}$) invariant. In particular, this is the case for the vectors $E(v)$  and respectively, $E^{*}(w)$. 

Recall that  $E$ is normal and different from zero. Thus, we have that $E^{*}E=EE^{*}\neq 0$ (otherwise $0=\left\langle E^{*}E (\phi), \phi \right\rangle =\Vert E (\phi) \Vert^{2}$ for every $\phi \in \mathcal{H}$ contradicting $E \neq 0$). In this way we are able to find a non-zero vector $\psi  \in E(\mathcal{H})\bigcap E^{*}(\mathcal{H})\neq \{0\} \subset \mathcal{H}$ which is $\left\langle N^{+}_{\alpha},N^{-}_{\alpha} \right\rangle $--invariant, so also $G$--invariant. The contradiction follows and the lemma stands proven.
\end{proof}

\begin{theorem}
\label{Crit_H_M}
Let $G$ be a second countable, locally compact group having the following properties:
\begin{list}{\roman{qcounter})~}{\usecounter{qcounter}}
\item
$G$ admits a decomposition $K_1A^{+}K_2$, where $K_1,K_2 \subset G$ are compact subsets and $A^{+}$ is an abelian sub semi-group of $G$;
\item 
every sequence $\alpha=\{g_{n}\}_{n>0} \subset A^{+}$, with $g_{n} \rightarrow \infty$,  admits a subsequence $\beta=\{g_{n_k}\}_{n_k>0}$ such that  $G= \overline{ \left\langle U^{+}_{\beta},U^{-}_{\beta}\right\rangle }$.
\end{list}

Then $G$ has  the Howe--Moore property. 
\end{theorem}
\begin{proof}
Notice that by Lemma~\ref{lem::sep_loc_compact} and the fact that $G$ is also separable, it is enough to prove the Howe--Moore property only for unitary representations that are over separable Hilbert spaces.

Therefore, let $(\pi, \mathcal{H})$ be a unitary representation of $G$, without non-zero $G$--invariant vectors and where $ \mathcal{H}$ a separable Hilbert space. Suppose that not all matrix coefficients of $(\pi, \mathcal{H})$ vanish at $\infty$. 

Using Lemma~\ref{sequn_A} there exists a sequence $\{ g_{n} \}_{n>0} \subset A^{+}$ and two non-zero vectors $v,w \in \mathcal{H}$ such that $g_{n} \rightarrow \infty$ and $ \vert  \left\langle \pi(g_{n})v, w \right\rangle \vert  \nrightarrow 0$. By passing to a subsequence $\{g_{n_{k}}\}_{k>0}$ and using the separability of $\mathcal{H}$ (see Lemma~\ref{vN_normal}), assume that $\{ \pi(g_{n_{k}}) \}_{n_{k}}$ weakly converges to an operator $E$ and that $\vert  \left\langle \pi(g_{n_k})v, w \right\rangle \vert >C>0$, for every $n_k$ and where $C$ is a constant.  Applying Lemma~\ref{lem::gen_G_invariant_vect} to a corresponding subsequence of $\alpha=\{ \pi(g_{n_{k}}) \}_{n_{k}} $ given by hypothesis ii), we obtain a contradiction.  Our criterion stands proven.
\end{proof}

\section{Unified proof of known examples}
\label{sec::unif_examples}
As announced, in this section we apply the criterion provided by Theorem~\ref{Crit_H_M} to the classes of groups mentioned in the introduction. Thereby we deduce, in a unified fashion, that the Howe--Moore property holds for all those classes.

\subsection{Real Lie groups}
\label{subsec::real_Lie_groups}

For the class of connected, non-compact, simple real Lie groups with finite center the proof is the same as the one given in Bekka--Mayer~\cite{BM}; still we outline the ideas.

It is a well-known fact that a connected, non-compact, semi-simple real Lie group $G$, with finite center, has a $KA^{+}K$ decomposition (called the polar decomposition) where $K$ is a maximal compact subgroup, $A < G$ is the closed subgroup corresponding to the maximal abelian sub-algebra of the Lie algebra of $G$ and $A^{+} \subset A$ is the corresponding positive sub semi-group part (see~\cite[Chapter~III]{BM} for more details).

\begin{example}
Let $G=\SL(n, \RR)$. Its maximal compact subgroup $K$, up to conjugation, is $\SO(n, \RR)$. The subgroup $A$ is the group of all diagonal matrices in $\SL(n, \RR)$ with positive entries and $A^{+}$ is the subset of $A$ formed by the elements $a \in A$ having the diagonal coefficients $a_{ii}$ verifying $a_{ii} \leq a_{i+1, i+1}$, for every $1\leq i \leq n-1$.
\end{example}

Therefore, to use our criterion Theorem \ref{Crit_H_M} for connected, non-compact, simple real Lie groups, with finite center, it is enough to prove that for every sequence $\alpha = \left( g_{n}\right)_{n \geq 1} \subset A^{+}$ with $g_{n} \rightarrow \infty$ one has that  $\overline{ \left\langle U^{+}_{\alpha},U^{-}_{\alpha}\right\rangle } =G$. More exactly, this is given by the  following general lemma (see~\cite[Chap.~III, Lemma~1.8]{BM}):

\begin{lemma}
\label{lem::existence_hyp_element}
Let $G$ be a connected, non-compact, semi-simple real Lie group, with finite center. Using the above notation, we have:
\begin{list}{\roman{qcounter})~}{\usecounter{qcounter}}
\item
For every $a \in A^{+}$, the subgroup $\overline{ \left\langle U^{+}_{a},U^{-}_{a}\right\rangle }$ is normal and non-discrete in $G$. 
\item 
For every $\alpha =\left( g_{n}\right)_{n \geq 1} \subset A^{+}$, with $g_{n} \rightarrow \infty$, there exist $b \in A^{+}$  and a subsequence $\beta=\{g_{n_k}\}_{n_k>0}$ of $\alpha$ such that $\overline{ \left\langle U^{+}_{\beta},U^{-}_{\beta}\right\rangle } =\overline{ \left\langle U^{+}_{b},U^{-}_{b}\right\rangle } $; moreover $N^{\pm}_{b}= N^{\pm}_{\beta}$.
\end{list}
\end{lemma}
The proof of this lemma uses the Lie algebra of $G$ and the root system.

\subsection{Euclidean building groups}

Let us now turn our attention to the totally disconnected analogs of semi-simple real Lie groups; namely locally compact groups acting continuously and properly by automorphisms and strongly transitively on locally finite thick Euclidean buildings. For the definition of a building and the related concepts the reader can freely consult the book of  Abramenko and Brown~\cite{AB}.

\begin{notation}

To fix the notation, let us denote a locally finite thick Euclidean building by $\Delta$. Fix for what follows an apartment $\mathcal{A}$ in $\Delta$,  a chamber $C$ in $\mathcal{A}$ and  a special vertex $x_0$ of $C$. By $\Ch(\mathcal{B})$ we denote  the set of all chambers of an apartment $\mathcal{B}$ in $\Delta$ and by $\partial \mathcal{B}$ its corresponding apartment in the spherical building at infinity $\partial \Delta$.
\end{notation}

We make use also of the following:

\begin{definition}
\label{def::strong_trans_action}
Let $\Delta$ be a Euclidean or a spherical building and $G$ be a subgroup of the full automorphisms group $\Aut(\Delta)$ of $\Delta$. We say that $G$ acts  \textbf{strongly transitively} on $\Delta$ if for any two pairs $(A_1,c_{1})$ and $(A_2,c_{2})$ consisting of an apartment $A_i$ and a chamber $c_i \in \Ch(A_i)$, there exists $g \in G$ such that $g(A_1)=A_2$ and $g(c_{1})=c_{2}$. Moreover, since buildings are colorable (i.e., the vertices of any chamber are colored differently and any chamber uses the same set of colors), we say that $g \in \Aut(\Delta)$ is \textbf{type-preserving} if $g$ preserves the coloration of the building. We say that $G \leq \Aut(\Delta)$ is \textbf{type-preserving} if all elements of $G$ are type-preserving.

\end{definition}

\begin{remark}
\label{rem::HM_type_pres}
Recall that the subgroup of all type-preserving automorphisms of $\Delta$ is of finite index in $\Aut(\Delta)$ (see Abramenko--Brown~\cite[Proposition~A.14]{AB}). Therefore, if a subgroup of $\Aut(\Delta)$ enjoys the Howe--Moore property then, by Proposition~\ref{prop::open_compact}, that group must be type-preserving. 
\end{remark}

For the rest of this article, by Remark~\ref{rem::HM_type_pres}, we design $G$ to be a closed, strongly transitive and type-preserving subgroup of $\Aut(\Delta)$. For most of such groups $G$ we want to verify the two conditions given in Theorem~\ref{Crit_H_M}. The first one is the well-known polar decomposition whose proof is recalled in the next subsection. 

\subsubsection*{Polar decomposition}
\label{subsubsec::Polar_decom}

Let $\Stab_{G}(\mathcal{A})$, respectively $\Fix_{G}(\mathcal{A})$, be the stabilizer subgroup, respectively the pointwise stabilizer subgroup, in $G$ of an apartment $\mathcal{A}$ of $\Delta$. 
\medskip

Recall that $\Stab_{G}(\mathcal{A}) / \Fix_{G}(\mathcal{A})$ is the Weyl group $W$, corresponding to the Euclidean building $\Delta$ and that $W$ is generated by the reflexions through the walls of a chamber $C \in \Ch(\mathcal{A})$. Another well-known fact is that $W$ contains a maximal abelian normal subgroup isomorphic to $\mathbb{Z}^{m}$, whose elements are Euclidean translation automorphisms of the apartment $\mathcal{A}$ and $m$ is the Euclidean dimension of $\mathcal{A}$. Denote this maximal abelian normal subgroup by $A$ and remark that its elements are induced by hyperbolic automorphisms of $\Delta$. In particular, every element of $A$ can be lifted (not in a unique way) to a hyperbolic element of $G$. 

\begin{remark}
\label{rem::algabraic_groups_local_fields}
Notice that, for a general closed, strongly transitive and type-preserving subgroup $G$ of $\Aut(\Delta)$, its corresponding abelian subgroup $A < W$ does not necessarily lift to an abelian subgroup of $G$. Still, this is the case if we consider that $G$ is a semi-simple algebraic group over a non Archimedean local field. By the theory of Bruhat--Tits, such a group gives rise to a Euclidean building where the group $G$ acts by automorphisms and strongly transitively. The abelian group $A <W$ is in fact the restriction to $\mathcal{A}$ of a maximal split torus of $G$, which is abelian.
\end{remark}

Let us now construct a basis of $A$ and define the sub semi-group $A^+$. For this, let $x_0 \in \mathcal{A}$ be a special vertex and let $Q$ be a sector in $\mathcal{A}$ based at $x_0$. Recall that $W_{x_0}$, the stabilizer subgroup in $W$ of the vertex $x_0$, is transitive on the set of sectors of $\mathcal{A}$ that are based at $x_0$. This implies that for every $a \in A$, there exists $k \in W_{x_0}$ such that $ka k^{-1}$ maps $x_0$ in $Q$. Moreover, as the Euclidean dimension of $\mathcal{A}$ is $m$ and $A$ is isomorphic to $\mathbb{Z}^{m}$, one can choose a basis $\{\gamma_1,\cdots, \gamma_m\}$ of $A$ such that $\gamma_i$ maps $x_0$ in $Q$, for every $i \in \{1,\cdots, m\}$. Define now $A^{+} \subset A$ to be the sub semi-group consisting of translations mapping $x_0$ in $Q$. Remark also that every element of $A^{+}$ is a product of positive powers of elements from the basis $\{\gamma_1,\cdots, \gamma_m\}$. The subset $A^{+}$ is the desired `maximal abelian' and `positive' sub semi-group part of $A$.

Using the above notation we recall below the proof of the \textbf{polar decomposition}.

\begin{lemma}
\label{lem::polar_decom}
Let $G$ be a closed, strongly transitive and type-preserving subgroup of $\Aut(\Delta)$. Take $\mathcal{A}$ an apartment in $\Delta$ and $x_0$ be a special vertex in $\mathcal{A}$. Then $G=\Stab_{G}(x_0) A^{+} \Stab_{G}(x_0)$.
\end{lemma}

\begin{remark}
We mention that the equality $G=\Stab_{G}(x_0) A^{+} \Stab_{G}(x_0)$ is an abuse of notation, as $A$ is not a subgroup of $G$. In fact, the polar decomposition is coming from the Bruhat decomposition of $G$ with respect to the affine BN-pair. The reader can consult the book of Abramenko and Brown~\cite[Sections~6.1,~6.2, Proposition~6.3]{AB}.
\end{remark}

\begin{proof}[Proof of Lemma~\ref{lem::polar_decom}]
Take $g \in G$ and let $\mathcal{B}$ be an apartment in $\Delta$ such that $C, g(C) \in \Ch(\mathcal{B})$. As $G$ is strongly transitive, there exists $k \in G$ such that $k(\mathcal{B})=\mathcal{A}$ and $k(C)=C$ pointwise. Therefore $k \in \Stab_{G}(x_0)$. As $ G$ is type-preserving, $kg(x_0)$ is a special vertex of the same type as $x_0$. Therefore, using the fact  that $W$ is simply transitive on special vertices of the same type, there exists a translation automorphism $a$ of the apartment $\mathcal{A}$ sending $kg(x_0)$ into $x_0$. We conclude that $akg(x_0)=x_0$ and thus $akg \in \Stab_{G}(x_0)$. By writing the element $a$ using elements from $A^{+}$ and $ \Stab_{G}(x_0)$, we have that $g \in \Stab_{G}(x_0) A^{+} \Stab_{G}(x_0)$.  Therefore the polar decomposition of the group $G$ is proven, as $\Stab_{G}(x_0)$, being the stabilizer of the vertex $x_0$, is compact and open in $G$. 
\end{proof}

\begin{example}[See Platonov--Rapinchuk~\cite{PR94} page~151]
For $G=\SL(n, \QQ_p)$ the `good' maximal compact subgroup is $\SL(n, \ZZ_{p})$ and $A$ is the group of diagonal matrices in $\SL(n, \QQ_p)$ of the form $\diag(p^{a_1}, \cdots, p^{a_n})$, with the condition that $(a_1,\cdots, a_n) \in \ZZ^{n}$ and $\sum\limits_{i=1}^{n} a_i=0$. In this case, the sub semi-group $A^{+}$ is the subset of $A$ such that $a_1 \leq a_2 \leq \cdots \leq a_n$.
\end{example}

\begin{example}(See Burger--Mozes~\cite[Section~4]{BM00b})
\label{exam::trees_B_M}
When the Euclidean building is a bi-regular tree $T_{bireg}$, with valence $\geq 3$ at every vertex, and $G$ is a closed, type-preserving subgroup of $\Aut(T_{bireg})$, with $2$--transitive action on the boundary of $T_{bireg}$, the polar decomposition goes as follows. Notice that $G$ acts without fixed point on the boundary of $T_{bireg}$ and from~\cite[Section~4]{BM00b}, we have that $G$ is not compact. By Tits~\cite{Ti70}, $G$ must contain a hyperbolic element. Moreover, by using the $2$--transitivity of $G$ on the boundary of $T_{bireg}$, we can construct a hyperbolic element $a \in G$ whose translation length is $2$. Denote by $\ell \subset T_{bireg}$ the  translation axis of $a$, which is a bi-infinite geodesic line of the tree $T_{bireg}$. Let $x_0$ be a vertex of $\ell$. The maximal compact subgroup of the polar decomposition of $G$ is the stabilizer subgroup in $G$ of $x_0$, the subgroup $A =\langle a\rangle$ and $A^{+}=\{a^{n} \; \vert \; n \geq 1\}$. We also mention the known fact that in the case of a thick tree $T$, the strongly transitive action of a subgroup of $\Aut(T)$ on $T$ is equivalent to the $2$--transitive action of that group on the boundary of $T$ (see for example Caprace--Ciobotaru~\cite[Theorem~1.1 and Corollary~3.6]{CaCi}).
\end{example}
 
Related to the polar decomposition from Lemma~\ref{lem::polar_decom}, we recall another well-known decomposition of the group $G$. 

Let $c$ be an ideal chamber in $\partial \mathcal{A}$. Denote by
\begin{equation}
\label{equ::geom_parab_subgroup}
G_c^{0}: =\{ g \in G \; \vert \; g(c)=c \text{ and } g \text{ fixes at least one vertex of } \Delta\}
\end{equation} the `pointwise' stabilizer subgroup in $G$ corresponding to $c \in \Ch(\partial \mathcal{A})$.  Remark that $G_c^{0}$ is a closed subgroup of $\Stab_{G}(c)$ and it does not contain any hyperbolic automorphisms of $\Delta$. Indeed, let $g \in G_c^{0}$ and let $x_g \in \Delta$ be a vertex fixed by $g$. Consider an apartment $\mathcal{B}$ in $\Delta$ such that $x_g \in \mathcal{B}$ and $c \in \Ch(\partial \mathcal{B})$. Then $\mathcal{A}$ and $\mathcal{B}$ share  a common sector pointing to the chamber at infinity $c$. As $g(c)=c$ and $g$ fixes the vertex $x_g$, $g$ must fix a sector in $\mathcal{A} \cap \mathcal{B}$ that points to the chamber $c$. Using this observation, it is easy to see that $G_c^{0}$ is a group. To prove that $G_c^{0}$ is closed, let $\{g_n\}_{n>0} \subset G_c^{0}$ such that $g_n \to g$, for some $g \in G$. We have that $g(c)=c$. Moreover, as $\{g_n\}_{n>0}$ are elliptic elements in $G$, $g$ cannot be hyperbolic. We have that $g \in G_c^{0}$ as desired.

Take $g \in G$ and let $\mathcal{B}$ be an apartment in $\Delta$ such that $g^{-1}(C) \in \Ch(\mathcal{B})$ and $c \in \Ch(\partial \mathcal{B})$. Remark that $\mathcal{A}$ and $\mathcal{B}$ have a sector in common, corresponding to the chamber at infinity $c$. Thus, by the strongly transitive action, there exists $n \in G_c^{0}$ such that $n(\mathcal{B})=\mathcal{A}$. We obtain that $ng^{-1}(C) \in \Ch(\mathcal{A})$. As $G$ is type-preserving, there exists a translation automorphism $a \in A $ of the apartment $\mathcal{A}$ such that $ang^{-1}(x_0)=x_0$. From here we conclude that $ang^{-1} \in \Stab_{G}(x_0)$ and thus $g \in \Stab_{G}(x_0) A G_c^{0}$. Therefore 
\begin{equation}
\label{equ::KAN_decom}
G=\Stab_{G}(x_0) A G_c^{0}.
\end{equation}

\subsubsection*{Levi conditions}
\label{subsubsec::Levi_cond}

Let us now start verifying the second condition given by Theorem~\ref{Crit_H_M}, for a closed, strongly transitive and type-preserving subgroup  $G$ of $\Aut(\Delta)$. Its first step is given by the following well-known proposition; for the convenience of the reader we include its proof.

Regarding the closed subgroup $G_c^{0}$ defined in (\ref{equ::geom_parab_subgroup}), by the strongly transitive action of $G$ on $\Delta$, we have that $G_c^{0}$ acts transitively on the set $\Opp(c)$ of all chambers opposite $c$ of $\Ch(\partial \Delta)$. Indeed, let $c_1 \neq c_2 \in \Opp(c)$. Denote by $A_1$ and respectively, by $A_2$, the unique apartments of $\Delta$ such that $c_1,c \in \Ch(\partial A_1)$ and respectively, $c_{2},c \in \Ch(\partial A_2)$.  Let $Q \subset A_1\cap A_2 $ be a sector pointing to the chamber at infinity $c$ and choose a chamber $C$ of $\Ch(\Delta)$ contained in the interior of $Q$. Then, by the strong transitive action of $G$, there exists $g \in G$ such that $g(C)=C$ and $g(A_1)=A_2$. In addition, we have that $g(Q)=Q$ pointwise; therefore $g \in G_c^{0}$ and $g(c_1)=c_2$.

\begin{proposition}
\label{prop::gen_parab_subgroups}
Let $G$ be a closed, strongly transitive and type-preserving subgroup of $\Aut(\Delta)$. Let $\mathcal{A}$ be an apartment of $\Delta$ and take $c _+\in \Ch(\partial \mathcal{A})$. Denote by $c_-$ the chamber opposite $c_+$ in $\partial \mathcal{A}$.

Then $G=\langle G_{c_+}^{0}, G_{c_-}^{0} \rangle$.
\end{proposition}

Before starting the proof, we state the next useful lemma.

\begin{lemma}
\label{lem::stab_ap_subgroup}
Let $G$ be a closed, strongly transitive and type-preserving subgroup of $\Aut(\Delta)$. Let $\mathcal{A}$ be an apartment of $\Delta$ and take $c _+\in \Ch(\partial \mathcal{A})$. Denote by $c_-$ the chamber opposite $c_+$ in $\partial \mathcal{A}$.

Then $\Stab_{G}(\mathcal{A})=\Stab_{G}( \partial \mathcal{A})  \leq  \langle G_{c_+}^{0}, G_{c_-}^{0} \rangle$.\end{lemma}

\begin{proof}
The equality $\Stab_{G}(\mathcal{A})=\Stab_{G}( \partial \mathcal{A})$ is straightforward.

 Following Caprace--Ciobotaru~\cite[proof of Lemma~3.5 ]{CaCi} we start with a basic observation. Let $\mathcal{A}'$ be an apartment of $\Delta$ such that the intersection $\mathcal{A}\cap \mathcal{A}'$ is a half-apartment. We claim that there is some $u \in \langle G_{c_+}^{0}, G_{c_-}^{0} \rangle $ mapping $\mathcal{A}'$ to $\mathcal{A}$ and fixing $\mathcal{A} \cap \mathcal{A}'$ pointwise. Indeed, because $c_+,c_-$ are opposite in $\partial \mathcal{A}$, they are separated by any wall of $\mathcal{A}$. Thus, $\mathcal{A} \cap \mathcal{A}'$, being a half-apartment, contains a sector $Q$ corresponding either to $c_+$, or to $c_-$. We apply the strong transitivity of $G$; there exists an element $g \in G$ fixing pointwise a chamber in the interior of the sector $Q$ and sending $\mathcal{A}'$ to $\mathcal{A}$. Then $g \in \langle G_{c_+}^{0}, G_{c_-}^{0} \rangle$.

\medskip

Let now $M$ be a wall of $\mathcal{A}$. Let $H$ and $H'$ be the two half-apartments of $\mathcal{A}$ determined by $M$. Since $\Delta$ is thick, there exists a half-apartment $H''$ such that $H \cup H''$ and $H' \cup H''$ are both apartments of $\Delta$. By the above claim, we can find an element $u \in \langle G_{c_+}^{0}, G_{c_-}^{0} \rangle$ fixing $H$ pointwise and mapping $H'$ to $H''$. Similarly, there are elements $v, w \in \langle G_{c_+}^{0}, G_{c_-}^{0} \rangle$ fixing $H'$ pointwise and such that $v(H'') = H$ and $w(H) = u^{-1}(H')$. Now we set $r := vuw$. By construction $r$ fixes pointwise the wall $M$ and $r \in \langle G_{c_+}^{0}, G_{c_-}^{0} \rangle$. Moreover we have
$$r(H) = vu u^{-1}(H') = v(H') = H'$$
and 
$$r(H') = vu(H') = v(H'') = H,$$
so that $r$ swaps $H$ and $H'$. It follows that $r$ stabilizes the apartment $\mathcal{A}$ and acts on $\mathcal{A}$ as the reflection through the wall $M$. Since this holds for an arbitrary wall of $\mathcal{A}$ and $\Stab_{G}(\mathcal{A}) /\Fix_{G}(\mathcal{A})$, being the Weyl group $W$, is generated by reflexions through the walls of a chamber $C \in \Ch(\mathcal{A})$, our conclusion follows. 
\end{proof}

\begin{proof}[Proof of Proposition \ref{prop::gen_parab_subgroups}]

Take $g \in G$ and let $\mathcal{B}$ an apartment of $\Delta$ such that $g(c_+), c_+ \in \Ch(\partial \mathcal{B})$. Then $\mathcal{A}$ and $\mathcal{B}$ have a sector in common, corresponding to the chamber at infinity $c_+$. By the strong transitivity of $G$, take $h_1 \in G_{c_+}^{0}$ such that $h_1(\mathcal{B})=\mathcal{A}$. Therefore $h_1g(c_+) \in \Ch(\partial \mathcal{A})$. Recall that the strong transitivity of $G$ on $\Delta$ implies that $G$ acts strongly transitively on $\bd \Delta$ (see Garrett~\cite[Section~17.1]{Gar97}). Applying Lemma \ref{lem::stab_ap_subgroup}, there exists $h_2 \in \Stab_{G}( \partial \mathcal{A}) \leq \langle G_{c_+}^{0}, G_{c_-}^{0} \rangle$ such that $h_2h_1g(c_+)=c_+ $. Thus $h_2h_1g(\mathcal{A})$ and  $\mathcal{A}$ have a sector in common corresponding to the chamber at infinity $c_+$. Take now $h_3 \in G_{c_+}^{0}$ with $h_3h_2h_1g(\mathcal{A})= \mathcal{A}$. Therefore $h_3h_2h_1g \in \Stab_{G}(\mathcal{A})$ and by Lemma \ref{lem::stab_ap_subgroup} the conclusion follows.
\end{proof}

\subsubsection*{`Algebraic' Levi decomposition}

Recall from Baumgartner and Willis~\cite{BW04} the following general theory. Let $G $ be a totally disconnected locally compact group and take $a \in G$. Let 
\begin{equation} 
\label{equ::parab_subgroup}
P^{+}_{a}:=\{ g \in G \; \vert \; \{a^{-n}g a^{n}\}_{n \in \mathbb{N}} \text{ is bounded}\}. 
\end{equation} In the same way, but using $a^{n}g a^{-n}$, we define $P^{-}_{a}$. Following~\cite[Section~3]{BW04}, $P^{+}_{a}$ is a closed subgroup of $G$ and denote
\begin{equation}
\label{equ::contr_subgroup}
U^{+}_{a}:=\{ g \in G \; \vert \; \lim_{n \to \infty }a^{-n}g a^{n}=e\}. 
\end{equation}
By~\cite[Section~3]{BW04}, $P^{+}_{a}$ and $U^{+}_{a}$ are called the \textbf{parabolic}, respectively the \textbf{contraction}, subgroups associated to $a$, where in general $U^{+}_{a}$ is not closed. Moreover, from the algebraic point of view, we have the following \textbf{Levi decomposition}:

\begin{theorem}(See Baumgartner and Willis~\cite[Proposition~3.4, Corollary~3.17]{BW04})
\label{thm::levi_decom}
Let $G$ be a totally disconnected locally compact group which is also metrizable. Let $a \in G$.
Then $P^{+}_{a}= U^{+}_{a} M_a$, where $M_{a}:= P^{+}_{a} \cap P^{-}_{a}$. Moreover, $U^{+}_{a}$ is normal in $P^{+}_{a}$.
\end{theorem}

Notice that Theorem~\ref{thm::levi_decom} applies to our context of $G$ being a closed, strongly transitive and type-preserving subgroup of $\Aut(\Delta)$, where $\Delta$ is a locally finite Euclidean building.


\begin{remark}
In the setting of a locally finite Euclidean building $\Delta$, a subset  $F \subset \Aut(\Delta)$ is called \textbf{bounded} if there exists a point $y \in \Delta$ such that the set $F \cdot y$ is a bounded subset of $\Delta$. For equivalent definitions see Abramenko--Brown~\cite[Lemma~11.32]{AB}. Those characterizations are useful  and used in what follows when working with elements of $P^{+}_{a}$.
\end{remark}

\subsubsection*{`Geometric' Levi decomposition}

This subsection proposes to link the `algebraic' Levi decomposition from Theorem~\ref{thm::levi_decom} with the subgroup $G_c^{0}$ defined in (\ref{equ::geom_parab_subgroup}), a key ingredient being given by the following:




\begin{proposition}
\label{prop::geom_levi_decom}
Let $G$ be a closed, non-compact and type-preserving subgroup of $\Aut(\Delta)$. Let $\mathcal{A}$ be an apartment in $\Delta$ and assume there exists a hyperbolic automorphism $a \in \Stab_{G}(\mathcal{A})$. Denote its attracting endpoint by $\xi_+ \in \partial \mathcal{A}$ and let  $c_+ \in \Ch(\partial \mathcal{A})$, with $\xi_+ \in c_+$. Let $G_{c_+}:= \{ g \in G \; \vert \; g(c_+)=c_+\}$ and $G_{\xi_+}:= \{ g \in G \; \vert \; g(\xi_+)=\xi_+\}$.

Then $G_{c_+} \leq P^{+}_{a} = G_{\xi_+}= G_{\sigma}$, where $\sigma$ is the unique simplex in $\bd \mathcal{A}$ such that $\xi_+$ is contained in the interior of $\sigma$ ($\sigma=\xi_+$ if and only if $\xi_+$ is a vertex). In particular, we obtain that $P^{+}_{a} \cap P^{-}_{a} = G_{\{\xi_-, \xi_+\}}$, where $G_{\{\xi_-, \xi_+\}}:=\{ g \in G \; \vert \; g(\xi_+)=\xi_+ \text{ and } g(\xi_-)=\xi_-\}$.
\end{proposition}

\begin{proof}

The equality $ G_{\xi_+}= G_{\sigma}$ is immediate as $G$ is type-preserving.

Let us first prove that $P^{+}_{a} \leq G_{\xi_+}$.

\medskip

Take $g \in P^{+}_{a}$ and let $x_0$ be a special vertex of $\mathcal{A}$. By the definition of $P^{+}_{a}$, the set $\{a^{-n}g a^{n}(x_0) \; \vert \; n \in \mathbb{N}\}$ is a bounded subset of $\Delta$ and thus is finite. Extract a subsequence $n_k \to \infty$ such that $a^{-n_k}g a^{n_k}(x_0)=y$, for every $n_k$. As $G$ is type-preserving,  the point $y$ is also a special vertex of $\Delta$. 

Now, because $a^{n_k}(x_0) \to \xi_+$ and $a^{n_k}(y) \to \xi_+$ (in the cone topology of $\Delta \cup \bd \Delta$), one can conclude that $g(\xi_+)=\xi_+$. Therefore $g \in G_{\xi_+}$. 

\medskip
Let us show that $G_{\xi_+} \leq P^{+}_{a}$.
\medskip

Let $g \in G_{\xi_+}$. As $g(\xi_+)=\xi_+$ there exists a geodesic ray in $\mathcal{A}$, say $[y, \xi_+)$, such that $g([y, \xi_+)) \subset \mathcal{A}$ is parallel to $[y, \xi_+)$. Moreover, we have that $\dist_{\mathcal{A}}(g(t),t) = \dist_{\mathcal{A}}(g(y),y)$, for every $t \in [y, \xi_+)$. To prove that $g \in P^{+}_{a}$ it is enough to show that the set $\{a^{-n}g a^{n}(z) \; \vert \; n \in \mathbb{N}\}$ is a bounded subset of $\Delta$, for some point $z \in \mathcal{A}$. Take $z=y$. Because $a$ is a hyperbolic element in $\Stab_{G}(\mathcal{A})$, $a^{n}(y) \in [y, \xi_+)$. From here, we immediately conclude that  $\{a^{-n}g a^{n}(z) \; \vert \; n \in \mathbb{N}\}$ is indeed a bounded set.

\medskip

The fact that $G_{c_+} \leq P^{+}_{a} $ is an immediate consequence of the equality $P^{+}_{a} = G_{\xi_+}$, as $G$ is type-preserving and $G_{c_+} \leq G_{\xi_+}$.
\end{proof}



\begin{remark}
For a general hyperbolic element $a \in \Stab_{G}(\mathcal{A})$ we do not necessarily have the equality $P^{+}_{a} =G_{c_+}$ in Proposition \ref{prop::geom_levi_decom}. However, if $a$ is a strongly regular hyperbolic element of $\Stab_{G}(\mathcal{A})$ (see Caprace--Ciobotaru~\cite[Definition~2.2 and Lemma~2.3 ]{CaCi}), the equality $P^{+}_{a} =G_{c_+}$ holds, as in this case $G_{c_+}=G_{\xi_+}$. In addition, if we take $\Delta$ to be a locally finite tree, the equality $P^{+}_{a} =G_{c_+}$ is always verified as the chambers at infinity are just the ends of the tree, which are points.
\end{remark}

By combining Proposition \ref{prop::geom_levi_decom} and Theorem \ref{thm::levi_decom} we obtain the \textbf{geometric Levi decomposition}:

\begin{corollary}
\label{coro::geom_levi_decom}
Let $G$ be a closed and type-preserving subgroup of $\Aut(\Delta)$. Let $\mathcal{A}$ be an apartment in $\Delta$ and assume there exists a hyperbolic automorphism $a \in \Stab_{G}(\mathcal{A})$. Denote its attracting endpoint by $ \xi_+ \in \partial \mathcal{A}$ and let $\sigma$ be a simplex in $\bd \mathcal{A}$ such that $\xi_+ \in \sigma$. Let also $c_+ \in \Ch(\partial \mathcal{A})$ with $\xi_+ \in c_+$.

Then $G_{\sigma}= U^{+}_{a} (M_a \cap G_{\sigma})$, where $M_{a}:= P^{+}_{a} \cap P^{-}_{a}$. In particular,  $G_{c_+}^{0}= U^{+}_{a} (M_a \cap G_{c_+}^{0})$. In addition $U^{+}_{a} $ is normal in $G_{c_+}^{0}$, respectively, $G_{\sigma}$.
\end{corollary}

\begin{proof}
By the definition of $U^{+}_{a}$ one can easily verify that $U^{+}_{a} \leq G_{c_+}^{0} \leq G_{\sigma} \leq P^{+}_{a}$. The conclusion follows by applying Theorem \ref{thm::levi_decom}. 
\end{proof}

Furthermore, by Proposition \ref{prop::gen_parab_subgroups} and Corollary \ref{coro::geom_levi_decom} we obtain:

\begin{corollary}
\label{coro::gen_parab_subgroups}
Let $G$ be a closed, strongly transitive and type-preserving subgroup of $\Aut(\Delta)$. Let  $\mathcal{A}$ be an apartment of $\Delta$ and let $a \in \Stab_{G}(\mathcal{A})$ be a hyperbolic automorphism.
 
Then $\overline{\langle U^{+}_{a},U^{-}_{a} \rangle}$ is normal in $G$. 
\end{corollary}

\begin{proof}
Let $c_-,c_+ \in \Ch(\bd \mathcal{A})$ be two opposite ideal chambers containing the repelling point and respectively the attracting point of the hyperbolic element $a$.  By Proposition \ref{prop::gen_parab_subgroups}  we have that $G=\langle G_{c_-}^{0}, G_{c_+}^{0}  \rangle$.  Therefore, to prove that $\overline{\langle U^{+}_{a},U^{-}_{a} \rangle}$ is normal in $G$, it is enough to verify that, for every $g \in G_{c_{\pm}}^{0}$, we have that $g U^{\pm}_{a}g^{-1} \in \langle U^{+}_{a},U^{-}_{a} \rangle$. Indeed, using the decomposition of $G_{c_{\pm}}^{0}$ given by Corollary \ref{coro::geom_levi_decom}, and the fact that $U^{\pm}_{a} $ is normal in $G_{c_\pm}^{0}$, the conclusion follows. 
\end{proof}

\subsubsection*{The proof of the Howe--Moore property}

Using all the ingredients presented in the above subsections, the goal here is to provide the proof of the following theorem:

\begin{theorem}
\label{thm::H-M-buildings}
Let $G$ be an isotropic simple algebraic group over a non Archimedean local field or a closed, topologically simple subgroup of $\Aut(T)$ that acts $2$--transitively on the boundary $\partial T$, where $T$ is a bi-regular tree with valence $\geq 3$ at every vertex. Then $G$ admits the Howe--Moore property.
\end{theorem}

To prove Theorem~\ref{thm::H-M-buildings}, in the case of algebraic groups, it remains to verify the second condition given by Theorem~\ref{Crit_H_M}, namely, the analogue of Lemma~\ref{lem::existence_hyp_element}. As in the case of connected simple real Lie groups with finite center, the analogue lemma uses the theory of root group datum related to Euclidean buildings; therefore, we briefly recall some definitions and state some known properties of this theory from Weiss~\cite{Weiss}, Abramenko--Brown ~\cite{AB} and Ronan~\cite{Ron89}. We mention that the proof of the analogue of Lemma~\ref{lem::existence_hyp_element}, stated in Proposition~\ref{prop::lem_analog} for the case of Euclidean buildings, works in the same spirit as for real Lie groups.

\begin{definition}
\label{def::root_group_sph}
Let $X$ be a Euclidean or a spherical building.  A half-apartment $h$ of $X$ is called a \textbf{root of $X$}. By abuse of notation, we use $\bd h$ to denote the boundary wall in $X$ which is determined by $h$. 

Moreover, if $X$ is a spherical building, the \textbf{root group} $U_{h}$ corresponding to a root $h$ of $X$ is defined to be the set of all $g \in \Aut(X)$ such that $g$ fixes  pointwise $h$ and  also the star of every panel contained in $h \setminus \bd h$. We denote $U_{h}^{*}:=U_{h} \setminus \{ \id\}$.
\end{definition}

\begin{example}
\label{exm::trees_roots}
If $X$ is a locally finite tree without vertex of valence one, then a root $h$ of $X$ corresponds to an infinite ray in the tree and its boundary $\bd h$ is just the base-vertex in $X$ of that infinite ray.  Denote by $\bd X$ the boundary at infinity of the tree $X$. For a point $\xi \in \bd X$, its corresponding root group is the contraction group $U^{+}_{a} \leq \Aut(X)$ defined in (\ref{equ::contr_subgroup}) above and where $a$ is a hyperbolic element of $\Aut(X)$ with $\xi$ its attracting endpoint (see Caprace--De Medts~\cite[Lemma 2.7]{CaDM}). 
\end{example}

\begin{definition}
\label{def::Moufang_prop}
Let $X$ be a thick spherical building. For a root $h$ of $X$, we denote by $\mathcal{A}(h)$ the set of all apartments in $X$ that contain $h$. We say that \textbf{$X$ is Moufang} if, for every root $h$ of $X$, the root group $U_{h} < \Aut(X)$ is transitive on $\mathcal{A}(h)$. 
\end{definition}

It is immediate that the boundary at infinity of a locally finite tree without vertex of valence one is Moufang.

Recall that for a semi-simple algebraic group $\mathbb{G}$ over a non Archimedean local field, by Bruhat and Tits~\cite{BT72} one associates a locally finite thick Euclidean building $\Delta$ where $\mathbb{G}$ acts by automorphisms and strongly transitively. More precisely:

\begin{definition}
\label{def::alg_groups}
Let $\mathbb{G}$ be a semi-simple algebraic group over a non Archimedean local field $k$, of rank $\geq 1$. Let $G=\mathbb{G}(k)$ be the group of all $k$--rational points of $\mathbb{G}$. By~\cite{BT72}, to the group $G$ one associates a locally finite thick Euclidean building $\Delta$ where $G$ acts by type-preserving automorphisms and strongly transitively. Moreover, by~\cite{BT72}, the spherical building $\bd \Delta$ at infinity of $\Delta$ is Moufang. Let $G^{+}:=  \langle U_h \leq G \; \vert \; h \text{ is a root of } \bd \Delta \rangle$. It is known that $G^{+}$ is a closed normal subgroup of $G$ and that the factor group $G/G^{+}$ is compact (see Margulis~\cite[Thm.~2.3.1]{Mar91}). We call $G^{+}$ \textbf{an isotropic simple algebraic group over the non Archimedean local field $k$}. From its definition, notice that $G^{+}$ acts by type-preserving automorphisms and also strongly transitively on $\Delta$. 
\end{definition}

\begin{proposition}(See Weiss~\cite[Prop.~13.2 and~13.5]{Weiss})
\label{prop::root_groups_at_infinity}
Let $\Delta$ be a Euclidean building such that its corresponding building at infinity $\bd \Delta$ is Moufang. Let $\mathcal{A}$ be an apartment of $\Delta$ and let $h$ be a root of $\bd \mathcal{A}$. For $u \in U_{h}^{*}$ let $h_u := \mathcal{A} \cap u(\mathcal{A})$. 

Then $h_u$ is a root of $\mathcal{A}$ such that its boundary at infinity $h_u(\infty)$ is $h$. Moreover, $u$ acts trivially on $h_u$. Conversely, let $\mathcal{H}$ be a root of $\mathcal{A}$ such that its boundary at infinity $\mathcal{H}(\infty)$ is $h$. Then there exists $u \in U_{h}^{*}$ such that $\mathcal{H}=h_u$.
\end{proposition}

\begin{definition}
\label{def::root_group_aff}
Let $\Delta$ be a Euclidean building such that its corresponding building at infinity $\bd \Delta$ is Moufang. Let $\mathcal{H}$ be a root of an apartment $\mathcal{A}$ of $\Delta$ and denote $h := \mathcal{H}(\infty)$ the corresponding root in $\bd \mathcal{A}$. With respect to $U_{h}$, the \textbf{affine root group corresponding to $\mathcal{H}$} is the set $U_{\mathcal{H}}:= \{u \in U_{h}^{*}\;\vert \; h_u=\mathcal{H}\} \cup \id$, where $h_u := \mathcal{A} \cap u(\mathcal{A})$, for $u \in U_{h}^{*}$.
\end{definition}

\begin{example}
\label{exm::trees_affine_roots}
Let $X$ is a locally finite tree without vertex of valence one as in Example~\ref{exm::trees_roots}. Let $\mathcal{H}$ be an infinite ray of $X$ and denote by $h \in \bd X$ the endpoint at infinity of $\mathcal{H}$. Then $U_{\mathcal{H}}:= \{g \in U_{a}^{+}  \; \vert \; h_g=\mathcal{H} \} \cup \id$, where $U_{a}^{+}$ is defined like in Example~\ref{exm::trees_roots}.
\end{example}

\begin{proposition}(See Bruhat--Tits~\cite[Prop.~7.4.33]{BT72})
\label{prop::contr_prop}
Let $\mathbb{G}$ be a semi-simple algebraic group over a non Archimedean local field $k$ and let $G=\mathbb{G}(k)$ be the group of all $k$--rational points of $\mathbb{G}$. Denote by $\Delta$ the corresponding locally finite thick Euclidean building on which $G$ acts by type-preserving automorphisms and strongly transitively.  Let $\mathcal{H}$ be a root of $\Delta$ and let $x \in \mathcal{H}$. Then the radius of the ball of  $\Delta$ around $x$ which is fixed pointwise by $U_{\mathcal{H}} < G$ goes to infinity as the distance from $x$ to $\bd \mathcal{H}$ goes to infinity. 
\end{proposition}

From Propositions~\ref{prop::contr_prop} and~\ref{prop::root_groups_at_infinity} we immediately obtain:

\begin{corollary}
\label{cor::contr_groups}
Let $\mathbb{G}$ be a semi-simple algebraic group over a non Archimedean local field $k$ and let $G=\mathbb{G}(k)$ be the group of all $k$--rational points of $\mathbb{G}$. Denote by $\Delta$ the corresponding locally finite thick Euclidean building on which $G$ acts by type-preserving automorphisms and strongly transitively.   Let $\mathcal{A}$ be an apartment of $\Delta$ and let $h$ be a root of $\bd \mathcal{A}$. Let $\alpha=\{a_n\}_{n>0} \subset \Stab_{G}(\mathcal{A})$ with $a_n \to \infty$. Assume that $a_n(x_0) \to \xi \in h \setminus \bd h$, in the cone topology, where $x_0 \in \mathcal{A}$ is some special vertex.

Then $U_{h} \leq U_{\alpha}^{+}$. In particular, we obtain that $U_{h} \leq U_{a}^{+}$, for every hyperbolic element $a  \in \Stab_{G}(\mathcal{A})$ whose attracting endpoint $\xi_+ \in h \setminus \bd h$.
\end{corollary}

\begin{proof}
As $U_h$ is the union of its corresponding affine root groups $U_{\mathcal{H}}$, it is enough to prove that $U_{\mathcal{H}} \leq U_{\alpha}^{+}$. Indeed, let $u \in U_{\mathcal{H}}$. As $a_n(x_0) \to \xi \in h \setminus \bd h$, in the cone topology, where $x_0$ is a special vertex of $\mathcal{A}$, we have that the intersection of the geodesic ray $[x_0, \xi)$ with the root $\mathcal{H}$ is a geodesic ray with endpoint $\xi$. Therefore, for every standard open neighborhood $V \subset \mathcal{H}$ of $\xi$ with respect to the cone topology of $\mathcal{A}$, there exists $N$ such that $a_n(x_0) \in V \subset \mathcal{H}$, for every $n \geq N$. By Proposition~\ref{prop::contr_prop}, we immediately obtain that $u \in U_{\alpha}^{+}$. 
\end{proof}

Another important property is recorded by the following corollary.

\begin{corollary}
\label{cor::equal_contr_root_groups}
Let $\mathbb{G}$ be a semi-simple algebraic group over a non Archimedean local field $k$ and let $G=\mathbb{G}(k)$ be the group of all $k$--rational points of $\mathbb{G}$. Denote by $\Delta$ the corresponding locally finite thick Euclidean building on which $G$ acts by type-preserving automorphisms and strongly transitively. Let $\mathcal{A}$ be an apartment of $\Delta$ and let $a \in \Stab_{G}(\mathcal{A})$ be a hyperbolic element whose attracting endpoint in $\bd \mathcal{A}$ is denoted by $\xi_+$. Let $\sigma$ be the unique simplex in $\bd \mathcal{A}$ such that $\xi_{+}$ is contained the interior of $\sigma$ and denote by $\St(\sigma)$ the star of $\sigma$ in $\bd \mathcal{A}$.

Then $U_{a}^{+}=  \langle U_h\; \vert \; h \text{ is a root of } \bd \mathcal{A} \text{ with } \St(\sigma) \subset h\rangle$. 
\end{corollary}

\begin{proof}
The conclusion follows by combining Corollary~\ref{coro::geom_levi_decom} with the theory of root group datum from Ronan~\cite[Thm.~6.18]{Ron89}.
\end{proof}

The analogue of Lemma~\ref{lem::existence_hyp_element} is given by the following proposition.
\begin{proposition}
\label{prop::lem_analog}
Let $\mathbb{G}$ be a semi-simple algebraic group over a non Archimedean local field $k$ and let $G=\mathbb{G}(k)$ be the group of all $k$--rational points of $\mathbb{G}$. Denote by $\Delta$ the corresponding locally finite thick Euclidean building on which $G$ acts by type-preserving automorphisms and strongly transitively.  Let $\mathcal{A}$ be the apartment that corresponds to the abelian sub semi-group $A^{+}$ of the Weyl group $W$ of $G$.

Let $\alpha=\{a_n\}_{n>0} \subset A^{+}$ with $a_n \to \infty$ and such that $a_{n}(x_0) \to \xi \in \bd \mathcal{A}$, for some special vertex $x_0 \in \mathcal{A}$. Then there exists a hyperbolic element $b \in \Stab_{G}(\mathcal{A})$ such that $U^{\pm}_{b} \leq U^{\pm}_{\alpha}$.
\end{proposition}

\begin{proof}
Let $x_0$ be a fixed special vertex in the apartment $\mathcal{A}$ and let $\{\gamma_1,\cdots, \gamma_m\}$ be the basis of $A^{+}$ described in the beginning of Section~\ref{subsubsec::Polar_decom}. Moreover, we can choose the basis $\{\gamma_1,\cdots, \gamma_m\}$ such that the attracting endpoint of $\gamma_j$, for every $j \in \{1, \cdots, m\}$, is a vertex in $\bd \mathcal{A}$ (i.e., it is a vertex of the chamber at infinity determined by the sector that defines $A^{+}$). Let $\sigma$ be the unique simplex of $\bd \mathcal{A}$ that contains $\xi$ in its interior ($\xi = \sigma$ if and only if $\xi$ is a vertex in $\bd \mathcal{A}$). Let  $\{\gamma_{i_1},\cdots, \gamma_{i_l}\}$ be the set of the elements of $\{\gamma_1,\cdots, \gamma_m\}$ whose attracting endpoints determine the simplex $\sigma$.  

Define $b:= \gamma_{i_1}\cdot... \cdot \gamma_{i_l}$ and notice that $b$ is a hyperbolic element in $\Stab_{G}(\mathcal{A})$. Moreover, the attracting endpoint $\xi_{+}$ of $b$ and the point $\xi$ are contained in the interior of the simplex $\sigma$ of $\bd \mathcal{A}$.
 
Apply Corollary~\ref{cor::equal_contr_root_groups} to $b$ and $\sigma$. Then, by Corollary~\ref{cor::contr_groups} we obtain that $U^{\pm}_{b} \leq U^{\pm}_{\alpha}$. 
\end{proof}

\begin{proof}[Proof of Theorem~\ref{thm::H-M-buildings}]
By Lemma~\ref{lem::sep_loc_compact} and Remark~\ref{rem::exam_sep_loc_com} it is enough to verify the Howe--Moore property only for separable Hilbert spaces. 

Let $G$ be an isotropic simple algebraic group over a non Archimedean local field or a closed, topologically simple subgroup of $\Aut(T)$ that acts $2$--transitively on the boundary $\partial T$, where $T$ is a bi-regular tree with valence $ \geq 3$ at every vertex. For such $G$, we verify the two conditions of Theorem~\ref{Crit_H_M}. In both cases, the first condition of Theorem~\ref{Crit_H_M} is verified by applying Lemma~\ref{lem::polar_decom} to $G$.

Let us verify the second condition of Theorem~\ref{Crit_H_M}. Notice that, in the case of a bi-regular tree, the corresponding group $A^{+}$ of $G$, given by the polar decomposition of $G$, is generated by a single hyperbolic element $a \in G$ (see Example~\ref{exam::trees_B_M}). Therefore, the second condition follows immediately from Corollary~\ref{coro::gen_parab_subgroups} applied to $G$. 

When $G$ is an isotropic simple algebraic group over a non Archimedean local field, to obtain the Howe--Moore property, it is enough to verify the second condition of Theorem~\ref{Crit_H_M} only for a sequence $\alpha=\{a_n\}_{n>0} \subset A^{+}$ with $a_n \to \infty$ and such that $a_{n}(x_0) \to \xi \in \bd \mathcal{A}$, for some special vertex $x_0 \in \mathcal{A}$. The latter assertion follows from the fact that $\bd \Delta \cup \Delta$ is compact with respect to the cone topology.
Then, apply to $\alpha=\{a_n\}_{n>0}$ Proposition~\ref{prop::lem_analog} and Corollary~\ref{coro::gen_parab_subgroups}. Notice that, for every hyperbolic element $b \in \Stab_G(\mathcal{A})$, the group $\overline{\langle U^{+}_{b},U^{-}_{b} \rangle}$, which is normal in $G$, acts non-trivially on $\Delta$. By Tits~\cite[Main Theorem]{Ti64}, we obtain that for every hyperbolic element $b \in \Stab_G(\mathcal{A})$, the group  $\overline{\langle U^{+}_{b},U^{-}_{b} \rangle}$ equals $G$. The theorem stands proven.
\end{proof}

\bigskip

\noindent
{\bf Acknowledgements.} \quad
This article is part of author's PhD project conducted under the supervision of Pierre-Emmanuel Caprace. I would like to thank him for valuable conversations and support during this project. My  thanks go also to Mihai Berbec and Steven Deprez for interesting discussions, in the early stage of this work, on von Neumann algebras, as well as providing the elegant proof of Lemma~\ref{vN_normal} and the Example~\ref{exa::not_normal}. I also thank the referee for his/her useful comments.

\end{document}